\documentclass[amsthm,amscd,amssymb,verbatim,epsf,amsmath,amsfonts]{amsart}

\usepackage[colorlinks=true,linkcolor=blue,citecolor=blue]{hyperref}
\usepackage{enumerate}
\usepackage{amsaddr}
\usepackage{xcolor}
\usepackage{amsmath}
\usepackage{ulem}
 \newtheorem{theor}{\indent\sc Theorem}[section]
 \newtheorem{corol}[theor]{\indent\sc Corollary}
 \newtheorem{lemma}[theor]{\indent\sc Lemma}   
 \newtheorem{prop}[theor]{\indent\sc Proposition}
 \newtheorem{example}[theor]{\indent\sc Example}
 
 \newtheorem{dfn}[theor]{\indent\sc Definition}

\newcommand{\real }{{\mathbb R}}
\newcommand{\complex }{{\mathbb C}}

\newcommand{\sitter}{{\mathbb S}_{1}^{3}(1)}

\newcommand{\mink }{{{\real}_{1}^{4}}}

\newcommand{\nlight}[1]{(1+{#1}{\overline{{#1}}},{#1}+{\overline{{#1}}},
-i({#1}-{\overline{{#1}}}),-1+{#1}{\overline{{#1}}})}

\newcommand{\weis}[2]{(1 + {#1}\overline{{#2}},{#1} + \overline{{#2}},-i ({#1} - \overline{{#2}}),-1 + {#1}\overline{{#2}})}

\newcommand {\lpr}[2]{{\langle}{#1},{#2}{\rangle}} 

\newcommand {\clpr}[2]{{\langle}{#1},{#2}{\rangle}^{\complex}}

\newcommand {\B}{{\mathcal B}}

\newcommand {\F}{{\mathcal F}}

\def\Re{{\frak R\frak e}}
\def\Im{{\frak I\frak m}}

\def\l{{\langle}}
\def\r{{\rangle}}

\newcommand{\grass}{\complex P^{3}}

\begin{document}

\title{Timelike surfaces in the de Sitter space $\sitter \subset \mink$}

\author {M. P. Dussan,\,A. P. Franco Filho} 
 \address{ Universidade de S\~ao Paulo, Departamento de
Matem\'atica - IME \\CEP: 05508-090. S\~ao Paulo. Brazil}
\email{dussan@ime.usp.br (M.P. Dussan),apadua@ime.usp.br (A.P. Franco Filho)}
\author{M. Magid}
\address{Wellesley College, Wellesley MA, 02181}
\email{Corresponding authour:  mmagid@wellesley.edu}

\maketitle



\begin{abstract} This paper studies timelike minimal  surfaces in the De Sitter space $\mathbb S^3_1(1) \subset \mathbb R^4_1$ via a complex variable. Using complex analysis
and stereographic projection of  lightlike vectors we obtain a representation formula. Real and complex special quadrics in $\mathbb CP^3$ are identified with the grassmannians of spacelike and timelike oriented 2-planes of $\mathbb R^4_1$, and the normal frame is written in terms of certain complex valued functions $x$ and $y$, which may be considered holomorphic functions as a special case. Then several results describing the analytic restrictions via solutions of certain PDE in complex variable, are shown. Finding solutions allows us to identify explicitly the representation of the associated surfaces.  Moreover, using our technique we find a new kind of complex
function which we call quasi-holomorphic  and which satisfy
a generalized version of  the Cauchy-Riemann equations. Our technique allows the explicit construction of many families of minimal timelike surfaces in $\mathbb S^3_1(1)$ whose intrinsic Gauss map will also belong to the same class  of surfaces. 
 \end{abstract}

\vspace{0.2cm}

\vspace{0.1cm}

\section{Introduction}
There have been many papers on timelike minimal surfaces in different ambient spaces.  One of the first is Louise McNertney's thesis (\cite{Be}) in 1980, followed, in 1990 by the work of Van de Woestyne (\cite{VdW}).  These papers work with either isotropic (null)  coordinates or isothermal coordinates and examine various differential equations to analyze timelike minimal surfaces.  Other techniques appear later. Beginning with the work of Konderak (\cite{K}) in 2005 we find the split-complex (para-complex) numbers used in place of complex numbers to extend some results from positive definite surfaces to timelike minimal surfaces.  This led, for example, to looking at the Bj\"orling problem for timelike surfaces in various ambient spaces; see for instance \cite{CDM}, \cite{DM}, \cite{DPM}, \cite{MO}.  While using the split-complex numbers allows many arguments to carry over to the timelike case, there are some difficulties - namely that all split-meromorphic functions have singularities that consist of curves, not points.  

Our main goal in this paper it is to re-introduce complex analysis into the study of timelike minimal surfaces using  parameterizations of the null cone and spacelike planes. In particular, our focus is  timelike minimal surfaces in the De Sitter space $S^3_1(1) \subset \mathbb R^4_1$ using a complex variable. In order to do this,  we associate, to two lightlike tangent vectors, an ordered pair $(x,y)$ obtained through  stereographic projection from the north pole, where $x$ and $y$ are functions defined on  open set of the surface and take complex values. Those functions may be assumed to be holomorphic functions when we restrict the conditions to obtain minimal surfaces. We also identify real and complex quadrics in $\mathbb C P^3$, respectively, with the set of  timelike or spacelike oriented planes of $\mathbb R^4_1$. This allows us to obtain a complex representation formula for the surface  involving the functions $x$ and $y$.  After that we establish our technique of constructing the minimal surface $S^3_1(1)$
by identifying the complex PDE  which appears when imposing the conditions  of flat normal bundle in $\mathbb R^4_1$ and the existence of isotropic coordinates (or lightlike coordinates) on the  surface. We call these, the spherical and isotropic conditions.  

Using the complex variable to study the timelike surfaces in $S^3_1(1)$, we also prove that if the surface $(M,f)$ is an isotropic surface in $S^3_1(1)$ where $f$ is represented in terms of $x$ and $y$, with intrinsic Gauss map  $\nu$, then the functions $x$ and $y$ satisfy a new type of  partial differential equation, which generalize the
Cauchy-Riemann equations. We call the solutions of that PDE, {\it quasi-holomorphic functions}. In particular, that set of solutions contains  the holomorphic functions. From a geometric point of view, we also show that  the pairs $(M, f)$ and $(M, \nu)$ are strongly related. More specifically,  if $(M,f )$ is assumed, for instance, minimal non-totally geodesic isotropic surface in $\mathbb S^3_1(1)$ with Gauss map $\nu (w)$, then $(M, \nu)$ will also represent an isotropic minimal surface in  $\mathbb S^3_1(1)$ non-totally geodesic with Gauss map $f(w)$, and conversely. Some explicit examples are given.

\vspace{0.1cm}
In particular, in order to  focus on minimal timelike surfaces in $\mathbb S^3_1(1)$ we assume that $x$ and $y$ are are holomorphic functions.  Then we prove that $x$ and $y$ are related through a Mobius map and that the argument $\theta$ of the integration factor of complex derivate $f_w$ has to be a harmonic function in $M$. Moreover we obtain the explicit expression of the $x$ and $y$ functions in terms of the argument $\theta$.  Finally we
use our technique to construct explicit families of minimal timelike surfaces in $\mathbb S^3_1(1)$ with the associated families of $(M,\nu)$.
\section{Preliminaries}

The Minkowski vector space $ \mink$ is the real vector space $\mathbb R^{4}$ endowed with the usual Euclidean topology and 
with the semi-Riemannian metric 
$$\langle \;, \; \rangle = -(d x^{1})^{2} + (d x^{2})^{2} + (d x^{3})^{2} + (d x^{4})^{2}.$$ 
It is oriented vectorially by $\partial_{1} \wedge \partial_{2} \wedge\partial_{3} \wedge\partial_{4}$ and temporally by $\partial_{1}$, where  $\{\partial_{1}, \partial_{2}, \partial_{3}, \partial_{4}\}$ is 
the canonical basis of $\mink$.

\vspace{0.1cm}
Throughout this paper, $M$ will be an open connected and simply connected subset of the set of the 
complex numbers $\complex$. We will denote by $\mathcal{H}(M)$ the set of holomorphic maps from $M \subset \complex$ 
into $\complex$. A map $f = P + i Q$ from $M$ into $\complex$ is an anti-holomorphic map if, and only if, its conjugate map
$\overline{f} = P - iQ$ is a holomorphic map. The set of all anti-holomorphic maps will be denoted by 
$\overline{\mathcal{H}}(M)$. The set of all continuously differentiable maps from $M$ into $\complex$ we will be denoted by 
${C}^{\infty}(M,\complex)$, and we say that these maps are smooth maps from $M$ into $\complex$.

Let
$$\frac{\partial}{\partial w} = \frac{1}{2} \left( \frac{\partial}{\partial u} - i \frac{\partial}{\partial v} \right) 
\; \; \ \ \ {\rm and } \; \; \ \ 
\frac{\partial}{\partial \overline{w}} = \frac{1}{2} \left( \frac{\partial}{\partial u} + i \frac{\partial}{\partial v} 
\right)$$
be the differential operators defined over the set of all smooth maps from $M$ into $\complex$, where $w = u + i v \in M$. 

It follows that 
a smooth map $f$ from $M$ into $\complex$ 
is a holomorphic map if and only if $\frac{\partial}{\partial \overline{w}} f(w) = 0$ for all $ w \in M$. 

Here we  will also use often  
the notation $\frac{\partial f}{\partial w} = f_{w}$ \ and \ $\frac{\partial f}{\partial \overline{w}} = f_{\overline{w}}$. 

\section{Surfaces in $\sitter$}

A parametric surface of $\mink$ is a two parameter function $f : M \longrightarrow \mink$ where 
$M$ is  a connected open subset of $\real^{2}$, satisfying the following conditions: \\
(1) The function $f$ is a homeomorphism from $M$ onto $S = f(M)$ endowed   subspace topology of $\mink$.\\
(2) The function $f$ is  $C^{\infty}(M,\mink)$. \\
(3) For each $w = (u,v) \in M$ the set $\{f_{u}(w),f_{v}(w)\}$ is a linearly independent set, and the induced metric is given by 
$$ds^{2}(f) = E du^2 + 2F du dv + G dv^2$$ 
where the functions $E(w)$, $F(w)$ and $G(w)$ are given by 
$$E(w) = \lpr{\frac{\partial f(w)}{\partial u}}{\frac{\partial f(w)}{\partial u}} \; \  \ {\rm and } \; \ \ \ 
F(w) = \lpr{\frac{\partial f(w)}{\partial u}}{\frac{\partial f(w)}{\partial v}} \; \ \ \ {\rm and } \;  \ \ \ 
G(w) = \lpr{\frac{\partial f(w)}{\partial v}}{\frac{\partial f(w)}{\partial v}}.$$
 
\begin{dfn}\label{1}
A timelike surface in the sphere $\sitter$ is the pair $(M,f)$, where the function $f:M \to \mathbb R^4_1$ satisfies the conditions 
(1),(2) and (3) above, and for each $w \in M$ we have $\lpr{f(w)}{f(w)} = 1$, with the  metric tensor satisfying  $EG-F^2<0$, i.e., it is a non-degenerate Lorentz metric.  We call the local coordinates {\rm null} or  {\rm isotropic} if the metric has the form:
$ds^{2}(f) = 2F du dv$. This is always possible locally to find null coordinates. 

{\rm In this paper  we call a surface isotropic when we are using these local null coordinates.}

\vspace{0.2cm}
We assume that the lightlike vectors fields $f_{u}$ and $f_{v}$ are future directed, hence, $F(w) < 0$ for each $w \in M$. Moreover, we assume the surface equipped with the  Gauss map $\nu : M \longrightarrow \sitter$ which  is defined by the 
following conditions:   for each $w \in M$,\\
(1) $\lpr{\nu(w)}{\nu(w)} = 1$ \ and \ $\lpr{\nu(w)}{f(w)} = 0$. \\ 
(2) $\lpr{\nu(w)}{f_{u}(w)} = 0 = \lpr{\nu(w)}{f_{v}(w)}$. \\ 
(3) The ordered set \ $\{f(w), f_{u}(w), f_{v}(w), \nu(w)\}$ is an oriented positive basis of \ $\mink$.  
\end{dfn}

We observe that, if we assume $(M,f)$ with $f:M \to \mathbb R^4_1$ and the Gauss map $\nu(w)$ as previously defined, it follows from
conditions (1) and (2) of Definition  \ref{1}, that $\nu_{u}(w)$ and  $\nu_{v}(w) \in T_{f(w)}S$. So, we call the condition
$$(\forall w \in M) \; \; \{\nu_{u}(w), \nu_{v}(w)\} \subset T_{f(w)}S $$
the {\it Spherical condition}. This means that the normal connection of this class of surface is flat. 

\subsection{Gauss and Weingarten equations} 

From now on we will assume that the Gauss map $\nu(w)$ is not constant. Next we will establish the Gauss and Weingarten equations 
for an isotropic surface $(M,f)$ of $\sitter$ with Gauss map $\nu(w)$. Let 
$$\B(w) = \{f(w), f_{u}(w), f_{v}(w), \nu(w)\}_{w \in M}$$ 
be the family of pointwise bases for $\mink$ given by (3) of Definition \ref{1}.

\begin{lemma}
Let $(M,f)$ be an isotropic surface of $\sitter$ equipped with the Gauss map $\nu(w)$. 
Since  $\nu_{u}(w), \nu_{v}(w) \in Span\{f_{u}(w),f_{v}(w)\}$, the structural equations for the surface are:
\begin{equation}
\left\{\begin{matrix}
f_{uu} = \frac{F_{u}}{F} f_{u} + a \nu \\
f_{uv} = -F f + b \nu \\
f_{vv} = \frac{F_{v}}{F} f_{v} + c \nu 
\end{matrix}\right. \; \; \ {\rm (Gauss), } \; \; \ \ \ \ 
\left\{\begin{matrix}
\nu_{u} =- \frac{b}{F} f_{u} - \frac{a}{F} f_{v} \\
\nu_{v} = -\frac{c}{F} f_{u} - \frac{b}{F} f_{v}
\end{matrix} \right. \; \; \ {\rm (Weingarten). }  
\end{equation}
Moreover, the surface $(M,f)$ is minimal if and only if $f_{uv}(w) = -F(w) f(w)$ and this means that $b(w) = 0$ for each $w \in M$.  
\end{lemma}
\begin{proof}We define $a=\l f_{uu}, \nu \r$, $b=\l f_{uv}, \nu \r$ and $c=\l f_{vv}, \nu \r$.  Once that is done it is easy to verify the Gauss and Weingarten equations.  For instance, since $\l f_u, f \r=0$ we have 
$$\l f_{uv},f \r + \l f_u, f_v \r=0,$$ thereby obtaining the coefficient of $f$ in the decomposition of $f_{uv}$.  Finally note that minimality means the trace of the shape operator is zero, or $b=0.$
\end{proof}

Note that when the Gauss map $\nu(w) \in \mink$ is a constant vector, the surface $f(M)$ is totally geodesic surface, 
hence it is a minimal surface of $\sitter$. 
The timelike hyperplane $[\nu]^{\perp}$ contain $S = f(M)$. The Gauss curvature of $S$ is $K(f)(w) = 1$ for all $w \in M$. 

\begin{corol}
Let $(M,f)$ be an isotropic surface of $\sitter$ equipped with the non-constant Gauss map $\nu(w)$. Then the fundamental  equations are
given by 
$$K(f) = \frac{-1}{F}\left(\frac{F_{u}}{F}\right)_{v} = 1 - \frac{ac - b^{2}}{F^{2}} 
\; \; \; \; {\rm  (Gauss)}$$ 
$$\frac{\partial b}{\partial u} - \frac{\partial a}{\partial v} = b \frac{F_{u}}{F} 
\; \; \ \  {\rm and } \; \; \ \ 
\frac{\partial b}{\partial v} - \frac{\partial c}{\partial u} = b \frac{F_{v}}{F} \; \; \; \; \ {\rm (Codazzi)}.$$ 

Moreover if $(M,f)$ is minimal then $a(u,v) = a(u)$ and $c(u,v) = c(v)$, that means $a$ and $c$ are functions which depend only of $u$ and $v$, respectively.
\end{corol}

\begin{proof}
The Gauss curvature equation follows from $\lpr{(f_{uu})_{v}}{f_{v}} = \lpr{(f_{uv})_{u}}{f_{v}}$. Hence, 
$$\left(\frac{F_{u}}{F}\right)_{v} \; F - \frac{ac}{F} \; F = -F \lpr{f_{u}}{f_{v}} + b \frac{-b}{F} F.$$ 
The Codazzi equations follows from $\lpr{(f_{uu})_{v}}{\nu} = \lpr{(f_{uv})_{u}}{\nu}$ and 
$\lpr{(f_{vv})_{u}}{\nu} = \lpr{(f_{uv})_{v}}{\nu}$. Indeed, 
$\lpr{(f_{uu})_{v}}{\nu} = a_{v} + b(F_{u}/F) = \lpr{(f_{uv})_{u}}{\nu} = b_{u}$ and 
$\lpr{(f_{vv})_{u}}{\nu} = b(F_{v}/F) + c_{u} = \lpr{(f_{uv})_{v}}{\nu} = b_{v}.$ 
\end{proof}

\begin{theor}\label{nu}
Let $(M,f)$ be an isotropic surface of $\sitter$ equipped with a non-constant Gauss map $\nu(w)$.
If the surface $f(M)$ is minimal then $(M,\nu)$ has {the same}
isotropic parameters and  is also minimal with
$$\nu_{u}(w) = \frac{-a(w)}{F(w)} f_{v}(w) \; \; \ \ {\rm and } \; \; \ \ \nu_{v}(w) = \frac{-c(w)}{F(w)} f_{u}(w).$$ 
 Moreover, the Gauss curvatures $K(f)$ of 
$f(M)$ and $K(\nu)$ of $\nu(S)$ are related by the equation: 
$$F^2 K(f) + ac K(\nu) = 0.$$ 
Hence, $(M,f)$ is flat if and only if $(M,\nu)$ is flat.
\end{theor}
\begin{proof}
Since $f_{u} = (-F/c) \nu_{v}$ and $f_{v} = (-F/a) \nu_{u}$, we see that $(M,\nu)$ is isotropic and minimal. 
If we let $\l \nu_u,\nu_v \r =\hat{F}$, so that the metric tensor of $(M,\nu)$ is $$ds^{2}(\nu) =  2\hat F du dv,$$   it follows that $\hat{F} = ac/F$. Now, from the Codazzi equations, 
we have that $a = a(u)$ and $c = c(v)$. Then taking the $v$ derivative of $(\log(\hat F))_u$ gives us
$$\left(\frac{\hat{F}_{u}}{\hat{F}}\right)_{v} = - \left(\frac{F_{u}}{F}\right)_{v} + \left(\frac{a_{v}}{a}\right)_{u} + 
\left(\frac{c_{u}}{c}\right)_{v} = - \left(\frac{F_{u}}{F}\right)_{v}.$$ 
Now, the formula $F^2 K(f) + ac K(\nu) = 0$ follows from the Gauss equations. 
\end{proof}

\subsection{The equation $\nu(w) = k f(w) + \vec{T}$} 
\begin{lemma}
Assume that the shape operator of the isotropic immersion $f$ is diagonalized but never zero. Then there is a constant vector $\vec{T}$ so that $\nu(w) = k f(w) + \vec{T}$. \end{lemma}
\begin{proof} We are assuming that   $a = 0 = c$ but $b \neq 0$ in  the Weingarten 
equations. From the Codazzi equations we have 
$$\frac{b_{u}}{b} = \frac{F_{u}}{F} \; \; \mbox{ and } \; \; \frac{b_{v}}{b} = \frac{F_{v}}{F} \; \; 
\mbox{ thus } \; \; \frac{b(w)}{F(w)} =  {-}k,$$ 
for some real number $k \neq 0$. Therefore, coming back to Weingarten equations we have
$$\nu_{u} = k f_{u} \; \; \mbox{ and } \; \; \nu_{v} =  k f_{v} \; \; \mbox{ thus } \; \; \nu(w) - k f(w) = \vec{T} \in \mink 
\setminus \{0\}.$$  
Note that $\vec{T}$ can not be $0$, because we are assuming that $\{f(w),\nu(w)\}$ is a pointwise orthonormal basis of the 
normal bundle of $(M,f)$. \end{proof}

The following example shows that there exist non-minimal surfaces $(M,f)$ and $(M,\nu)$ sharing isotropic parameters. 
  
\begin{example}
Let $X(w) = (X^{1}(w), X^{2}(w), X^{3}(w), 0)$ be an isotropic parametrization of an open subset of the sphere 
$\{X \in \sitter |  X^{4} = 0\}$. Defining for $\theta \in ]0, \pi/2[$ and $w \in M = dom(X)$
$$\nu(w) = \cos \theta \; \vec{e}_{4} - \sin \theta \; X(w) \; \; \mbox{ and } \; \; 
f(w) = \sin \theta \; \vec{e}_{4} + \cos \theta \; X(w),$$ 
we have for $k = - \tan \theta$ and $\vec{T} = \sec \theta \; \vec{e}_{4}$ a solution of the equation $\nu(w) = k f(w) + \vec{T}$. 
\end{example}

\begin{theor}
Let $(M,f)$ be an isotropic surface of $\sitter$ equipped with the non-constant Gauss map $\nu(w)$. 
If $(M,f)$ and $(M,\nu)$ are isotropic solutions of the equation $\nu(w) = k f(w) + \vec{T}$, then there exists a basis for $\mink$, $\{\vec{t}_{1}, \vec{t}_{2}, \vec{t}_{3}, \vec{t}_{4}\}$, 
for which $k = - \tan \theta$ and $\vec{T} = \sec \theta \vec{t}_{4}$ for some $\theta \in ]0,\pi/2[$. In addition, there exists the isotropic 
parametrization $(M,X)$ of the open subset $\{X = X^{1} \vec{t}_{1} + X^{2} \vec{t}_{2} + X^{3} \vec{t}_{3} |  \lpr{X}{X} = 1\}$ 
for which the solution of \  $\nu(w) = k f(w) + \vec{T}$ \ is 
$$\nu(w) = \cos \theta \; \vec{t}_{4} - \sin \theta \; X(w) \; \; \mbox{ and } \; \; 
f(w) = \sin \theta \; \vec{t}_{4} + \cos \theta \; X(w).$$ 

The Gauss curvatures are $K(f)  = sec^2 (\theta)$  and $K(\nu) = \csc^2(\theta)$.  
\end{theor}

Next we will give an example of a  timelike minimal surface with non-null Gauss curvature, together with a coordinate transformations which allows us to obtain an equivalent isotropic surface. This example is a type of Clifford torus for $\sitter$. The (unique) 
coordinate transformations also forces $(M,\nu)$ to have isotropic parameter by Theorem \ref{nu}.

\begin{example} 
Let 
$$c_1(t) = (\sinh t,0,0,\cosh t) \; \; \mbox{ and } \; \; c_2(s) = (0,\cos s, \sin s,0)$$ 
be two curves of $\sitter$, being the first a timelike curve and the second a spacelike curve. 
Taking the two-parameter map 
$$X(x,y) = \cos x \; c_1(y) + \sin x \; c_2(y)$$ 
we have $X_x = -\sin x \; c_(y) + \cos x \; c_2(y)$ and $X_y = \cos x \; c_1'(y) + \sin x \; c_2'(y)$. Thus the metric tensor is 
$E(x,y) = 1$ and $F(x,y) = 0$ and $G(x,y) = -\cos^{2} x + \sin^{2} x = - \cos 2x$.  
The unitary normal is given by: 
$$\nu(x,y) = \frac{1}{\sqrt{\cos 2 x}} (\sin x \; c_1'(y) + \cos x \; c_2'(y)).$$ 
Since $X_{xx} = - X$ and $X_{yy} = \cos x \; c_1(y) - \sin x \; c_2(y)$ and $X_{xy} = -\sin x \; c_1'(y) + \cos x \; c_2'(y)$, 
the second quadratic form is $\Psi_{ij} = \lpr{D_{ij}X}{X} X + \lpr{D_{ij}X}{\nu} \nu = X_{ij} X + N_{ij} \nu$ \ or  
$$[\Psi_{ij}] = \left[\begin{matrix}
-1 & 0 \\ 0 & \cos 2 x
\end{matrix} \right] X + \left[\begin{matrix}
0 & 1/\sqrt{\cos 2x} \\ 1/\sqrt{\cos 2x} & 0
\end{matrix} \right] \nu.$$ 
Therefore 
$$[\Psi_{i}^{j}] = \left[\begin{matrix}
-1 & 0 \\ 0 & -1
\end{matrix} \right] X + \left[\begin{matrix}
0 & -1/\sqrt{\cos^{3} 2x} \\ 1/\sqrt{\cos 2x} & 0
\end{matrix} \right] \nu.$$ 
$$H = trace(\Psi) = -X \; \; \mbox{ and } \; \; 
K(f) = \det(X_{i}^{j}) + \det(N_{i}^{j}) = 1 + \sec^{2} 2x.$$ 
Here $H$ is the mean curvature vector of the immersion into $\mink$.

Define the coordinate transformations $p = p(x)$ and $q(y) = y$ and take $Y(p,q) = X(x(p),y(q))$. Then we have that the metric coefficients for  $Y$ are given by

{$$ \overline{E}(p,q) = x'(p)^2, \; \; \overline{F}(p,q) = 0 \; \; \ \ {\rm and } \; \; \ \ \overline{G}(p,q) = -\cos(2x(p)),$$}
 
Take 
 $$\int \frac{dx}{\sqrt{\cos 2x}} = p(x).$$
Now setting $u = p + q$ and $v = p - q$ we obtain the equivalent surface $(M,f)$ where $f(u,v) = Y(p+q,p-q)$ 
equipped with isotropic parameters. 
\end{example}

\section{An integration problem} 
In this section we look for conditions which allows us to find a representation formula for the isotropic surfaces.
We start identifying local representations for   lightlike vectors $L$  which are in the tangent spaces.  Moreover we identify 
orthogonal complements of the tangent spaces  together the complex and real quadric of $\mathbb CP^3$ corresponding
to set of spacelike and timelike oriented planes of $\mathbb R^4_1$.

\vspace{0.2cm}
If $L = (L^{1},L^{2},L^{3},L^{4})$ is a future directed lightlike vector with $L^{1} > 0$, then there exists an unique vector $\vec n \in \real^{3} = Span\{\vec e_{2}, \vec e_{3}, \vec e_{4}\}$ such that 
$$L = L^{1}(\vec e_{1} + \vec n) \; \; \ \ {\rm where } \; \;  \ \ \vec n = (0, L^{2}/L^{1},L^{3}/L^{1},L^{4}/L^{1}).$$ 
Since $\lpr{L}{L} = 0$ we have $\lpr{\vec n}{\vec n} = 1$. Let $North = (0,0,0,1)$ and define  stereographic projection $st$, 
by 
$$st(L) = a + ib = \left(\frac{L^{2}/L^{1} + i L^{3}/L^{1}}{1 - L^{4}/L^{1}} \right) = \frac{L^{2} + i L^{3}}{L^{1} - L^{4}} 
\in \complex \cup\{\infty\}$$
where $st(L) = \infty$ if and only if $L = {\mu}(1,0,0,1)$. Moreover, $st(L) = 0$ if and only if $L =  {\mu}(1,0,0,-1)$, with  {$\mu>0$.}

\begin{prop}\label{39}
For each isotropic plane $Span\{L_{1},L_{2}\} \subset \mink$ there exists an unique ordered pair 
$(x,y) \in (\complex \cup\{\infty\})^{2}$, such that  we can express, for $\mu_1, \mu_2>0,$
$$\mu_1 L_{1} =\tilde L_1= \nlight{x} \; \; \mbox{so that } \; \; x = st(L_{1}),$$
$$\mu_2 L_{2} =\tilde L_2= \nlight{y} \; \; \mbox{so that } \; \; y = st(L_{2}).$$


Therefore $\lpr{\tilde L_{1}}{\tilde L_{2}} = -2 \vert x - y \vert^{2}.$ The map $\F$ from the set of oriented isotropic planes in the 
square of the Riemann sphere $(\complex \cup\{\infty\})^{2}$ given by 
$$\F(Span\{L_{1},L_{2}\}) = (st(L_{1}),st(L_{2}))$$ 
is one-to-one and onto the open subset  
$(\complex \cup\{\infty\})^{2} \setminus \{(x,x) | \ x \in \complex \cup\{\infty\}\}.$
\end{prop}
At this point, with a slight abuse of notation we define $$L(x)=\nlight{x}.$$
\subsection{The orthogonal complement $[L_{1},L_{2}]^{\perp} = [W] \in \complex P^{3}$} \ 

\vspace{0.2cm}
Let $\clpr{...}{...}$ be the natural extension of the Lorentz inner product to $\complex^{4}$
and  $\mink = T \oplus S$ be a direct sum of a timelike plane $T = Span\{L_{1},L_{2}\}$ and a spacelike plane 
$S = Span\{X,Y\}$, where we assume that \\
(1) the lightlike vectors $L_{1}$ and $L_{2}$ are future directed. \\
(2) the ordered set $\{X,L_{1},L_{2},Y\}$ is a positive basis of $\mink$ obeying the relations: 
$$\lpr{X}{X} = \lpr{Y}{Y} >0, \; \; \lpr{X}{Y} = 0, \; \; \lpr{X}{L_{i}} = 0 = \lpr{Y}{L_{i}} \; \; i = 1,2.$$ 

\vspace{0.2cm}
Next we define the Grassmannians of the spacelike oriented planes and timelike oriented planes of $\mink$ within the complex projective space 
$\complex P^{3} = \complex^{4}/_{\equiv}$ as follows. 

\vspace{0.1cm}
If $\mu = a + ib \neq 0$ is a complex number and $Z = X + i Y$ is the complex vector associated to the basis of the spacelike plane 
$S$, then $\mu Z = (aX - bY) + i(bX + aY)$ gives us another basis of $S$ satisfying the condition (2) above. 
By definition we have 
$[Z] = [X + i Y] = \{ \mu Z | \mu \in \complex \; \; \mbox{ and } \; \; \mu \neq 0\}$
are the  equivalence classes that define  points of $\grass$. Now, taking the complex vector 
$T = L_{1} + i L_{2}$ associated to a timelike plane, and a complex number $\mu = a + ib \neq 0$, 
we have the complex vector $A + i B = \mu T = (aL_{1} - b L_{2}) + i (b L_{1} + a L_{2})$ satisfying 
$$\lpr{A}{A} = -2ab \lpr{L_{1}}{L_{2}} = - \lpr{B}{B} \; \; \mbox{ and } \; \; 
\lpr{A}{B} = (a^{2} - b^{2}) \lpr{L_{1}}{L_{2}}.$$  
Therefore, $\{A,B\}$ is also  a basis of timelike plane $T,$ and the determinant of the matrix associated to this  basis is 
$- \lpr{L_{1}}{L_{2}}^2\vert \mu \vert^{2} < 0$. Then we define: 

\begin{dfn}
$$Q_{space} = \{[Z] \in \complex P^{3} | \clpr{Z}{Z} = 0 \;  \mbox{ and }  \; \clpr{Z}{\overline{Z}} > 0\},$$ 
the complex quadric of $\complex P^{3}$ of the set of spacelike oriented planes of $\mink$. 

$$Q_{time} = \{[Z] \in \complex P^{3} | \clpr{Z}{\overline{Z}} = 0 \;  \mbox{ and } \; \clpr{Z}{Z} \neq 0\},$$ 
the real quadric of $\complex P^{3}$ of the set of timelike oriented planes of $\mink$. 
\end{dfn} 

Now we will obtain a important correspondence between $Q_{space}$ and  $Q_{time}$.  First we consider homogeneous coordinates for $Q_{space}$. 

Given $x,y \in \complex$, with $x \neq y$, let  
\begin{equation}\label{defiw}
W(x,y) = \weis{x}{y} \in \complex^{4}.
\end{equation}

For $x \in \complex$ and $y = \infty$ or for $x = \infty$ and $y \in \complex$ we set 
$W(x,\infty) = (x,1,i,x)$ or $W(\infty,y) =(\overline{y},1,-i,\overline{y})$.

\begin{prop}
Given the isotropic plane $Span\{L_{1},L_{2}\},$ let $W(x,y)$ be the complex vector (\ref{defiw}). Then 
$\clpr{W(x,y)}{L_{1}} = 0$ and $\clpr{W(x,y)}{L_{2}} = 0$ if, and only if $x = st(L_{1})$ and $y = st(L_{2})$ or  $y = st(L_{1})$ and $x = st(L_{2})$. \ Moreover $$\clpr{W(x,y)}{W(x,y)} = 0 \; \; \mbox{ and }\; \; \clpr{W(x,y)}{\overline{W(x,y)}} = - \lpr{ L(x)}{ L(y)} = 
2 \vert x - y \vert^{2}>0.$$ 
Hence, there exists the bijection $\F : Q_{time} \longrightarrow Q_{space}$ with $$\F([L_{1} + i L_{2}]) = [W(x,y)].$$  
  
\end{prop}

\vspace{0.2cm}

\subsection{\bf An integration problem}

Let $M$ be a connected and simply connected open subset of $\complex$ and let $w = (u,v) = u + iv \in M$ denote its points. 
Given two smooth functions $A, B: M \to \real$, there exists another two smooth functions $a, b: M \to \real$ such that 
$\Gamma = a A du + b B dv$ is a closed $1$-form if and only if $a_{v} A - b_{u} B = -a A_{v} + b B_{u}.$  {$d\Gamma =0$ is the definition of closed.}

Since $M$ is assumed to be a connected and simply connected open subset, it follows that if the form $\Gamma$ is closed then there is a smooth function $\varphi: M \to \real$ such that $d \varphi = a A du + b B dv.$

\vspace{0.2cm}
Next we will apply this last fact to vector fields along $M$. First suppose that 
$$V(w) = (\varphi^{1}(w), \varphi^{2}(w), \varphi^{3}(w), \varphi^{4}(w))$$
is a smooth vector field along $M$ such that  $\{V_{u}(w),V_{v}(w)\}_{w \in M}$ is a set of lightlike vectors which is linearly independent. 
Therefore,  
there exist complex functions $x,y$ and real valued functions $\alpha, \beta$ such that 
$$ V_u(w) = \alpha(w) L(x(w)) \; \; \mbox{ and } \; \;  V_v(w)  = \beta(w) L(y(w)),$$ 
where $\lpr{L(x)}{L(y)} = -2\vert x - y \vert^{2} \neq 0$. In coordinates, if we take $L = (L^{1},L^{2},L^{3},L^{4})$ 
$$\Gamma^{i} = \frac{\partial \varphi^{i}}{\partial u} du + \frac{\partial \varphi^{i}}{\partial v} dv = 
\alpha L^{i}(x)du + \beta L^{i}(y)dv.$$ 

In words, we have a unique pair $\alpha$ and $\beta$ for each coordinate $1$-form $\Gamma^{i} = d \varphi^{i}$.  

\vspace{0.2cm}
Now, we assume that the vector $1$-form $\Gamma = \alpha L(x) du + \beta L(y) dv$ is given over the ring $\F(M,\real)$ 
of smooth functions from $M$ into $\real$. Since we are assuming that $M$ is a simply connected open subset of $\complex$, 
we have: 

\begin{prop}
The vector-valued $1$-form $\Gamma= \alpha L(x) du + \beta L(y) dv$ is exact if, and only if it is closed.
Then the following equation is a necessary and 
sufficient condition for the existence of the vector field $V(w)$ such that $dV = \Gamma$ 
\begin{equation}\label{3}
d \Gamma = \left[-\left(\alpha_{v} L(x) + \alpha \frac{\partial L(x)}{\partial v}\right) + 
\left(\beta_{u} L(y) + \beta \frac{\partial L(y)}{\partial u}\right)\right] du \wedge dv = 0.
\end{equation} 

If equation (\ref{3}) holds then the vector field $V(w)$ is given by:  
\begin{equation}
V(w) = V_{0} + \int_{0}^{w} \alpha L(x) du + \beta L(y) dv.
\end{equation} 

Moreover, from $\lpr{d \Gamma(\partial u,\partial v)}{L(y)} = 0$ and $\lpr{d \Gamma(\partial u,\partial v)}{L(x)} = 0$  follow the equations 
\begin{equation}\label{5}
\frac{1}{\alpha} \; \frac{\partial \alpha}{\partial v} = \frac{-\lpr{\partial_{v}L(x)}{L(y)}}{\lpr{L(x)}{L(y)}} \; \;  \ {\rm and } \; \; \
\frac{1}{\beta} \; \frac{\partial \beta}{\partial u} =  \frac{- \lpr{\partial_{u} L(y)}{L(x)}}{\lpr{L(x)}{L(y)}}.
\end{equation} 
The equation (\ref{5}) is a necessary condition, but it is not sufficient.
\end{prop}  

\begin{proof}
Starting with  $\lpr{d \Gamma(\partial u,\partial v)}{L(y)} = 0$ 
 we have 
\begin{eqnarray*}\frac{\alpha_{v}}{\alpha}=-\frac{\l(L(x))_v,L(y)\r}{\l L(x),L(y)\r}=\\
-\frac{\l ((x \bar x)_v, x_v+ \bar x_v, -i(x_v-\bar x_v), (x \bar x)_v),(1 + y\bar y, y + \bar y , -i(y-\bar y),-1 + y\bar y)\r}{-2(x-y)(\bar x - \bar y)}=\\
\frac{-x_{v}}{x - y} + \frac{-\overline{x}_{v}}{\overline{x} - \overline{y}}.
\end{eqnarray*}
The same proof works for $\beta$,
so that equations (\ref{5})  become
\begin{equation}\label{7}
\frac{\alpha_v}{\alpha} =  \frac{- x_v}{x-y} +  \frac{- \overline x_v}{\overline x- \overline y} \ \ \  \ \ {\rm and} \ \ \  \ \ 
\frac{\beta_u}{\beta} =  \frac{y_u}{x-y} +  \frac{\overline y_u}{\overline x-\overline y}.
\end{equation}
\end{proof}

\section{Constructing timelike parametric surface in $\sitter$}

Let us take $W(x,y)$ given by equation (\ref{defiw}), where 
$$x(w) = st(f_{u}(w)) \; \; \mbox{ and } \; \; y(w) = st(f_{v}(w))$$ 
and $(M,f)$ is an isotropic surface of $\sitter$ equipped with the non-constant Gauss map $\nu(w)$. 
Then we find a map $\mu(x,y) \in \complex$ for which $f(w)$ is given by the following equation:
\begin{equation}\label{standardequation}
f(w) = \frac{\mu \; W(x,y) + \overline{\mu} \; \overline{W(x,y)}}{2} \; \; \mbox{and } \; \; |\mu|^2\clpr{W(x,y)}{\overline{W(x,y)}} =2.
\end{equation} 

\vspace{0.2cm}

Next we look for complex partial differential equations which relate the functions $\mu(w)$, $x(w)$ and $y(w)$ for $(M,f),$ where $f(w)$ is the map given by equations (\ref{standardequation}), and such that its Gauss 
map $\nu(w)$ has the following form:   
\begin{equation}\label{normal}
\nu(w) = \frac{\mu \; W(x,y) - \overline{\mu} \; \overline{W(x,y)}}{2i}, 
\end{equation}
satisfying that $(\forall w \in M) \; \; \{\nu_{u}(w), \nu_{v}(w)\} \subset T_{f(w)}S.$  We seek those partial differential equations whose solution will guarantee that $(M,f)$ is a parametric surface of $\mathbb S^3_1(1)$ whose Gauss map is  exactly the function $\nu(w)$.  
This means we are looking for the  spherical conditions for equation (\ref{standardequation}). We recall that $\clpr{W(x,y)}{W(x,y)} = 0 = \clpr{\overline{W(x,y)}}{\overline{W(x,y)}}.$

\begin{lemma}[Spherical conditions]\label{67}
Let $f(w)$ be the map given by equations (\ref{standardequation}) with $x,y, \mu \in \F(M,\complex)$ and $W(x(w),y(w))$ given by equation (\ref{defiw}).  Let $\nu(w)$ be the map given by equation (\ref{normal}). 
Then, $(M,f)$ is a parametric surface of a scaled \ $\sitter$ equipped with Gauss map $(M, \nu)$ if, and only if, the following equations
\begin{align}\label{spherical}
\frac{\mu_{w}}{\mu} = - \frac{\clpr{W_{w}}{\overline{W}}}{\clpr{W}{\overline{W}}} \; \; \  \ {\rm and } \; \; \ \ 
\frac{\mu_{\overline{w}}}{\mu} = - \frac{\clpr{W_{\overline{w}}}{\overline{W}}}{\clpr{W}{\overline{W}}}.
\end{align}
are satisfied.   
\end{lemma}

\begin{proof}
From equation (\ref{standardequation}) we have $\mu \bar \mu \l W, \bar W \r ^{\mathbb C} = 2$ \  hence we have 
$$\frac{\mu_{w}}{\mu} + \frac{\clpr{W_{w}}{\overline{W}}}{\clpr{W}{\overline{W}}} + 
\frac{\overline{\mu}_{w}}{\overline{\mu}} + \frac{\clpr{\overline{W}_{w}}{W}}{\clpr{W}{\overline{W}}} = 0. $$ {{Since $\nu(w)$ is the Gauss map it follows that for all $w \in M$, \ $ \{\nu_{u}(w), \nu_{v}(w)\} \subset T_{f(w)}S.$}}  As we saw above,   $\l f_w, W\r^{\mathbb C} = 0 = \l f_w, \overline W\r^{\mathbb C}$. So, 
$$
\mu_w \l W, \overline W\r = - \mu \l W_w, \overline W\r \ \ \  {\rm and} \ \ \overline \mu_w \l W, \overline W \r = - \overline \mu \l \overline W_w, W\r.
$$
 Equations (\ref{spherical}) follow from these equations.

Now if equations (\ref{spherical}) are satisfied then $\mu \bar \mu \l W, \bar W \r^{\mathbb C} = c >0$, hence
 $\l f, f\r = constant>0$. Since $2 f_w = (\mu W)_w + (\overline \mu \overline W)_w$
then from equation (\ref{spherical}) it follows that 
$$
\l f_w, W\r^{\mathbb C}= \frac{1}{2}[ \overline \mu_w\l \overline W, W\r + \overline \mu \l \overline W_w, W\r] = 0$$
$$
\l f_w, \overline W\r^{\mathbb C} = \frac{1}{2}[ \mu_w\l \overline W, W\r + \mu \l W_w, \overline W\r] = 0.$$
Therefore for all $w \in M$, $ \{\nu_{u}(w), \nu_{v}(w)\} \subset T_{f(w)}S. $ So $(M,f)$ is a parametric surface of a scaled $\mathbb S^3_1(1)$ with Gauss map $\nu(w)$.
\end{proof}

Next we look for the conditions which imply that we can choose the parametric coordinates to be isotropic at every point of $M$.
 
\begin{lemma}[Isotropic condition]\label{68}
Let $(M,f)$ and $(M,\nu)$ be the maps given respectively by (\ref{defiw}) and (\ref{standardequation}), for which equations (\ref{spherical}) hold. \ Then the pair $(M,f)$ is a parametric isotropic surface of $\sitter$ with Gauss map $(M,\nu)$ if and only if 
the following equations 
\begin{equation}\label{2}
\begin{cases}
\Im (\mu \clpr{W_{w}}{L(y)} + \overline{\mu} \clpr{\overline{W}_{w}}{L(y)}) = 0 \\
\Re (\mu \clpr{W_{w}}{L(x)} + \overline{\mu} \clpr{\overline{W}_{w}}{L(x)}) = 0
\end{cases}
\end{equation}
are satisfied.
\end{lemma}
\begin{proof}
From hypothesis we are taking $W(x,y)$ such that $x= st(f_u(w))$ and $y =st(f_v(w))$. Hence we have that $f_u(w) = \alpha L(x)$ and
 $f_v(w) = \beta L(y)$ for $\alpha, \beta$ real-valued functions. \ Since $\l f_w, L(y)\r$ is {{real}}  valued, and 
 $2 \l f_w, L(y)\r = \mu \l W_w, L(y)\r + \overline \mu \l \overline W_w, L(y)\r $, it follows that $\Im( \mu \l W_w, L(y)\r  + \overline \mu \l\overline W_w, L(y)\r) = 0.$ In similar way since $\l f_w, L(x)\r$ is imaginary valued, the second equation of (\ref{2}) follows. 

\vspace{0.1cm}
We  now show sufficiency. The map $f(w)$ is given, and (8) says that $\nu(w)$ is its Gauss map, then, we have a 
timelike surface of $\sitter$. A pointwise isotropic basis for the tangent bundle $T_{f(w)}S$, by Proposition (\ref{39}) 
is given by $\{L(x(w)),L(y(w))\}_{w \in M}$. Then we need to show that $f_u$ and $f_v$ are isotropic. In fact, since $f_{u} = f_{w} + f_{\overline{w}} = A L(x) + B L(y)$ and $f_{v} = -i(C L(x) - D L(y))$, the first and second equation 
in (9) implies respectively that $C(w) = 0$ and  $B(w) = 0$ for all $w \in M$. 
\end{proof}

So,  the pairs $(M,f)$ and $(M, \nu)$ given above, are strongly related. In fact if $(M,f)$  is assumed, for instance,  to be a minimal non-totally geodesic isotropic surface in $\mathbb S^3_1(1)$ with Gauss map $\nu(w)$, then $(M,\nu)$ will also represent an isotropic minimal surface in $\mathbb S^3_1(1)$ which is non-totally geodesic with Gauss map $f(w)$, and conversely. In fact

\begin{theor}\label{99}
Let $(M,f)$ be a minimal parametric isotropic surface given respectively by  (\ref{standardequation}) equipped with Gauss map given by (\ref{normal}). Then, $(M, \nu)$ is also a minimal non-totally geodesic isotropic surface in  $\mathbb S^3_1(1)$ with Gauss map $f(w)$.
Moreover, the isotropic condition for $(M,\nu)$ is given by the equations  
\begin{equation}
\begin{cases}\label{40}
\Im(\mu \clpr{W_{w}}{L(y)} - \overline{\mu} \clpr{\overline{W}_{w}}{L(y)}) = 0 \\
\Re(\mu \clpr{W_{w}}{L(x)} - \overline{\mu} \clpr{\overline{W}_{w}}{L(x)}) = 0.
\end{cases}
\end{equation}  
\end{theor}
\begin{proof}
Since $(M,f)$ is minimal, by the Weingarten equations we have that
$\nu_u = \frac{-a}{F} f_v$ and $\nu_v = \frac{-c}{F} f_u$. Hence
$$
\nu_w = \frac{1}{2} (\frac{-a \beta}{F} L(y) + i \frac{c \alpha}{F} L(x)),
$$
where $f_u(w) = \alpha L(x), \ f_v(w) = \beta L(y)$, since by hypothesis $x$ and $y$ are chosen such that $x = st(f_u(w))$ and   
$y = st(f_v(w))$. 

We see easily that $Span\{f_u, f_v\} \subset T_{\nu(w)} S$.\hskip .1in
  So $(M,\nu)$ is a isotropic surface in $\mathbb S^3_1(1)$ with Gauss map given by $(M,f)$, which is also minimal non-totally geodesic. 
Moreover, the isotropic condition for $(M,\nu)$ are obtained as follows.  We have  $$\nu_w = \frac{1}{i} ((\mu W)_w - f_w) = \frac{1}{i} (-(\overline \mu \overline W)_w + f_w).  $$Since  $\lpr{\nu_{w}}{L(x)}$ is real valued then
$\l (\mu W)_w - (\overline \mu \overline W)_w, L(x)\r$ is pure imaginary, then this corresponds to $\Re(\mu \l W_w, L(x)\r - \overline \mu \l \overline W_w, L(x) \r)=0$. So the second equation of (\ref{40}) is obtained. Similarly, the first equation is gotten using the fact that
$\l \nu_w, L(y) \r$ is now pure imaginary. 
\end{proof}

\subsection{A complex basis} Let us take the set of complex vectors 
$$c_{1} = (1,0,0,-1), \; \; \; c_{2} = (0,1,-i,0), \; \; \; c_{3} = (0,1,i,0), \; \; \; c_{4} = (1,0,0,1).$$
Each vector of this set, is null for the bilinear form $\clpr{}{}$, and the matrix of $\clpr{c_{i}}{c_{j}} = C_{ij}$ 
is given by 
$$C_{ij} = \left[\begin{matrix} 0 & 0 & 0 & -2 \\ 0 & 0 & 2 & 0 \\ 0 & 2 & 0 & 0 \\ -2 & 0 & 0 & 0 \end{matrix} \right].$$
 
\vspace{0.2cm}
In this special basis we have 
\begin{equation}\label{70}
L(x) = c_{1} + x c_{2} + \overline{x} c_{3} + x\overline{x} c_{4}, \ \ \ \ \ \ W(x,y) = c_{1} + x c_{2} + \overline{y} c_{3} + x\overline{y} c_{4},
\end{equation}
and we easily see that 
$W(y,x) = \overline{W(x,y)}$. This basis makes many of our computations easier. For example, 
 if $x = x(w)$ and $y = y(w)$, then 
$$W_ w = x_{w}(c_{2} + \overline{y} c_{4}) + \overline{y}_{w}(c_{3} + x c_{4}) \  \ \ {\rm and} \ \  \ \l W_w, L(x) \r ^{\mathbb C}  = 2(\overline{x} - \overline{y})x_{w}.$$

We observe that using the above basis, the spherical conditions (\ref{spherical}) given by Lemma (\ref{67}) are equivalent to
\begin{equation}\label{92}
\vert \mu \vert = \frac{1}{\vert x - y \vert} \; \; \mbox{ and } \; \; 
\frac{\mu_{w}}{\mu} = \frac{-x_{w}}{x - y} + \frac{\overline{y}_{w}}{\overline{x} - \overline{y}} \; \; \mbox{ and } \; \; 
\frac{\mu_{\overline{w}}}{\mu} = \frac{-x_{\overline{w}}}{x - y} + 
\frac{\overline{y}_{\overline{w}}}{\overline{x} - \overline{y}}.
\end{equation} 
 Furthermore, the isotropic condition (\ref{2}) given by Lemma (\ref{68}), with the orientation given by $\{L(x(w)),L(y(w))\}$ are equivalent to
\begin{equation}\label{89}
\frac{\overline{\mu} y_{v}}{x - y} + \frac{\mu \; \overline{y}_{v}}{\overline{x} - \overline{y}} = 0\ \ \mbox{ and } 
\ \ \frac{\mu x_{u}}{x - y} + \frac{\overline{\mu} \; \overline{x}_{u}}{\overline{x} - \overline{y}} = 0 \; \; 
\end{equation}

\subsection{Formulas for mean curvature of timelike parametric surfaces in $\sitter$}

\vspace{0.2cm}
Recall that we are assuming that $(M,f)$ is an isotropic surface in $\sitter \subset \mathbb R^4_1$. Thus, there exists two smooth functions $\alpha, \beta: M \to \mathbb R$ and two smooth functions $x,y : M \to \mathbb C$ such that
$$
f_u(w) = \alpha(w) L(x(w)) \ \ \ \  {\rm and} \ \ \ \   f_v(w) = \beta(w) L(y(w)),
$$
and the metric is such that $F= \l f_u, f_v\r = -2 \alpha \beta |x-y|^2.$ \ Moreover there exists also a smooth complex function $\mu: M \to \mathbb C$ such that
$f(w)$, $W(x,y)$ are given by formulas (\ref{standardequation}), (\ref{defiw}),  and the intrinsic Gauss map is the function $\nu(w)$ given by formula (\ref{normal}).
We have also the fixed   reference frame $\mathcal B = \{f(w), L(x(w)), L(y(w)), \nu(w))\}$.

\vspace{0.1cm}

The mean curvature of this surface is the trace of $A_\nu=\l H_f, \nu \r=\l \frac{f_{uv}}{F}, \nu \r,$ where $F= \l f_u, f_v\r$ and $H_f$ is the mean curvature vector. \ We will write this in the form $\frac{\Phi(w)}{F}$, where  $\Phi = \lpr{f_{uv}}{\nu}$.

\vspace{0.2cm}
Next we will study  $\Phi$. In fact, since $\Phi = \lpr{f_{uv}}{\nu} = \lpr{(\alpha L(x))_{v}}{\nu} = -\alpha \lpr{L(x)}{\nu_{v}}$ we have that 
$$\Phi = -\alpha \lpr{L(x)}{\nu_{v}} = 
-\alpha \lpr{L(x)}{(\mu W_{v} - \overline{\mu} {\overline{W}_{v}})/2i}.$$
Using formula (\ref{70}) we get $ \lpr{L(x)}{W_{v}} = 2(\overline{x} - \overline{y}) x_{v}$ and  $\lpr{L(x)}{\overline{W}_{v}} = 2(x - y) \overline{x}_{v}.$
Thus, 
\begin{equation}
\Phi = -2\frac{\alpha}{2i}(\mu(\overline{x} - \overline{y})x_{v} - \overline{\mu}(x - y) \overline{x}_{v}) = 
- 2 \alpha \Im(\mu(\overline{x} - \overline{y})x_{v}). 
\end{equation} 

Again since  $\Phi = \lpr{f_{uv}}{\nu} = \lpr{(\beta L(y))_{u}}{\nu}$ we have that 
$$\Phi = -\beta \lpr{L(y)}{\nu_{u}} = 
-\beta \lpr{L(y)}{(\mu W_{u} - \overline{\mu} {\overline{W}_{u}})/2i} = 2 \beta \Im(\mu(x - y)\overline{y}_{u}).$$ 
 Altogether then we have:
\begin{equation}\label{min}
\Phi = 2 \beta \Im(\mu(x - y)\overline{y}_{u}) \; \; \ {\rm and } \; \; \  \Phi = -2 \alpha \Im(\mu(\overline{x} - \overline{y})x_{v}). 
\end{equation} 

Hence we have the next result. 

\begin{lemma} 
Let $(M,f)$ be an isotropic parametric surface of the de Sitter space $\sitter$. With the notation above we have 
$$\alpha \Im(\mu(\overline{x} - \overline{y})x_{v}) + \beta \Im(\mu(x - y)\overline{y}_{u}) = 0.$$
\end{lemma}

Now we continue looking by formulas for $F$, and for the functions  $\alpha$ and $\beta$.

\begin{lemma} 
Let $(M,f)$ be an isotropic parametric surface of the de Sitter space $\sitter$. Assume that $\Phi(w)/F$ 
is the intrinsic mean curvature of $S = f(M)$. Then
\begin{equation}\label{g03}
F = -2 \alpha \beta \vert x - y\vert^{2} = 2\alpha \Re(\mu(\overline{x} - \overline{y})x_{v}) = - 2\beta \Re(\mu(x - y)\overline{y}_{u}),
\end{equation}
and therefore: 
\begin{equation}\label{alpha,beta}
\alpha = \Re\left(\mu \frac{\overline{y}_{u}}{\overline{x} - \overline{y}} \right) \ \ \; \; {\rm and } \; \;  \ \ 
\beta = - \Re\left(\mu \frac{x_{v}}{x - y} \right).
\end{equation} 
In particular if \ $\Phi = 0$ \ then the real valued functions $\alpha$ and $\beta$ become to 
\begin{equation}\label{alpha,beta minimal}
\alpha = \mu \frac{\overline{y}_{u}}{\overline{x} - \overline{y}} \; \ \  \; {\rm and } \; \;  \ \ 
\beta = - \mu \frac{x_{v}}{x - y}.
\end{equation}
\end{lemma}

\begin{proof}
Since $\lpr{f_{uv}}{f} = -\lpr{f_{u}}{f_{v}} = 2 \alpha \beta \vert x - y \vert^{2}$, we obtain 
$$-F = \frac{1}{2} \left(\lpr{f_{uv}}{\mu W(x,y)} + \lpr{f_{uv}}{\overline{\mu} W(y,x)}\right) = 
\alpha \left(\frac{\mu}{2} \lpr{L_{v}(x)}{W(x,y)} + \frac{\overline{\mu}}{2} \lpr{L_{v}(x)}{\overline{W}}\right)$$ 
then, equation (\ref{g03}) follows from $F = 2\alpha \Re(\mu x_{v}(\overline{x} - \overline{y}))$.   In the same way, $F = -2 \beta \Re({\mu} (x-y)\overline{y}_{u})$.  The equation (\ref{alpha,beta}) follows by substitution, and  (\ref{alpha,beta minimal})  then from (\ref{min}).
\end{proof}

\section{When $\Phi = 0$ and a new class of functions}

In this section we continue under the same conditions as in Section 5 and  focus on the case when $\Phi =0$. In fact we start with next result.
\begin{theor} 
If $(M,f)$ is an isotropic parametric surface of $\sitter$ with mean curvature vector $H_{f}$, then 
\begin{equation}\label{mc}
\lpr{H_{f}}{\nu} = \frac{1}{\beta} \Im\left(\mu \frac{x_{v}}{x - y} \right) = 
\frac{1}{\alpha} \Im\left(\overline{\mu} \frac{y_{u}}{x - y}\right).
\end{equation} 
Moreover, if $\lpr{H_{f}}{\nu} = 0$, so that $\Phi =0$, then 
\begin{equation}\label{xuvyuv}
x_{uv} = \frac{2 x_{u} x_{v}}{x - y} \; \; \ \ {\rm and } \; \ \ \; y_{uv} = \frac{-2 y_{u} y_{v}}{x - y}. 
\end{equation}
\end{theor}
\begin{proof}
First note that (\ref{mc}) follows from (\ref{min}) and (\ref{g03}), using the fact that $\Im \gamma =-\Im (\overline \gamma).$ Next we show the equations in (\ref{xuvyuv}). Taking the logarithmic derivative of equation (\ref{alpha,beta minimal}) for the real valued function $\beta$ we obtain
$$\frac{\beta_{u}}{\beta} = \frac{\mu_{u}}{\mu} + \frac{x_{uv}}{x_{v}} - \frac{x_{u} - y_{u}}{x - y}.$$ 
From second part of equation (\ref{7}) for $\beta$ and from the version of equation (\ref{spherical}) for the variable $u$, namely 
$$\frac{\mu_{u}}{\mu} = \frac{-x_{u}}{x - y} + \frac{\overline{y}_{u}}{\overline{x} - \overline{y}},$$ 
we finally get  the first equation of (\ref{xuvyuv}). The second equation follows in a similar way.
\end{proof}

\begin{theor}\label{91}
Let $(M,f)$ be an isotropic parametric surface in $\sitter$ such that $\Phi = 0$ and 
$$f(w) = \frac{\mu(w) W(x(w),y(w)) + \overline{\mu(w)} W(y(w),x(w))}{2},$$ with $f_u=\alpha L(x)$ and $f_v=\beta L(y)$.
Then the functions $x,  y: M \to  \complex$ belong a class of complex function $Z(w) = \varphi(w) + i \psi(w)$ 
such that 
\begin{equation}
\frac{\partial Z}{\partial v} = i \sigma(w) \frac{\partial Z}{\partial u}, \; \; \mbox{ where } \; \; \sigma : M \longrightarrow \real \; \; 
\mbox{with  } \; \; \sigma(w) \neq 0 \; \; (\forall w \in M). 
\end{equation} 
Moreover it follows the following equations type Cauchy-Riemann  
\begin{equation}\label{90}
\left\{ \begin{matrix}
\varphi_{u} = \frac{1}{\sigma} \; \psi_{v}\\
\varphi_{v} = -\sigma \psi_{u}. 
\end{matrix} \right.
\end{equation}
\end{theor}
\begin{proof}
Assuming that $\Phi = 0$, we get, from second equation of (\ref{min}) that
$$\frac{\mu x_{v}}{x - y} = \frac{\overline{\mu} \; \overline{x}_{v}}{\overline{x} - \overline{y}}.$$ 
Then taking this last equation together with the second equation of (\ref{89}), it follows that $x_{u} \; \overline{x}_{v} + \overline{x}_{u} \; x_{v} = 0.$  Then writing $x = a + i b$, from this last equation, we obtain $a_{u} \; a_{v} + b_{u} \; b_{v} = 0$ 
which means that the set of $\real^{2}$-vectors $\{(b_{v},-a_{v}),(a_{u},b_{u})\}$ is a linearly dependent set. 
This last equation says that, pointwise, there exists a real valued function $\sigma = \sigma(u,v)$ such that 
$$x_{v}(u,v) = i \sigma(u,v) \; x_{u}(u,v) \; \; \mbox{  for } \; \; (u,v) \in M.$$ 
An analogous computation shows that the function $y = y(w)$ satisfies $y_{u}(u,v) = i \xi(u,v) \; y_{v}(u,v)$ for some real 
valued function $\xi = \xi(u,v)$ defined over $M$.  Then we get equation (\ref{90}).\end{proof}

Using the content of Theorem (\ref{91}) we define a new class of functions, as follows.
\begin{dfn}
A complex function $Z : M \longrightarrow \complex$ is defined {\rm quasi-holomorphic} if, and only if, there exists a real valued 
function $\sigma : M \longrightarrow \real$ such that 
$$\frac{\partial Z}{\partial v} = i \sigma \frac{\partial Z}{\partial u}.$$ 
We denote this set of functions  by $\mathcal{O}(M)$.  Observe that $\sigma = 1$ implies that $Z$ is holomorphic function on $M$, which means that $Z' = Z_{u}$ and $Z' = -i Z_{v}$.
\end{dfn}

In particular we have the following subsets
\begin{prop}
The class of holomorphic and anti-holomorphic functions, $\mathcal{H}(M)$, $\overline{\mathcal{H}(M)}$, are 
contained in the class $\mathcal{O}(M)$, which is closed under conjugation $\mathcal{O}(M) = \overline{\mathcal{O}(M)}$.    
\end{prop}

\begin{example}
Let $Z(w) = \varphi(u,v) + i \psi(u,v)$ be a holomorphic function. Taking two real valued functions $a(u)$ and $b(v)$ we define 
the function 
$$\Psi = \varphi(a(u),b(v)) + i \psi(a(u),b(v)),$$  
which belongs  to $\mathcal{O}(M)$.
Indeed, since $\Psi_{u} = a'(u) Z_{u}$ and $\Psi_{v} = b'(v) Z_{v}$, since $Z \in \mathcal{H}(M)$, \ it follows from $Z_{v} = i Z_{u} = iZ'$, for   
$\sigma = b'(v)/a'(u)$, that $\Psi_{v} = i \sigma \Psi_{u}$. 

For instance, if we take $Z(w)=w^2$,  $a(u)=u$ and $b(v)=v^2$.  This gives $\Psi(u,v) = u^{2} - v^{4} + 2 i u v^{2} \in \mathcal{O}(M)$ and $\sigma(u,v) = 2v$. Indeed, 
$\Psi_{v} = 2vi(2u + 2iv^2) = 2vi\Psi_{u}$. 
\end{example}

\begin{example}
We observe that a solution of the system (\ref{xuvyuv}) is given by the real valued functions $x = v$ and $y = u$. Then, we take the parametric surface 
$$f(u,v) = \frac{W(v,u) + W(u,v)}{2(u - v)} \; \; \ \ {\rm for } \; \; \ \ M = \{(u,v) \in \complex | \; u > v\}.$$ 
Since in this case $\mu = 1/(u - v)$,  the spherical condition (\ref{92}) and isotropic conditions (\ref{89}) are satisfied trivially.  
Furthermore, since the third coordinate $f^{3}(u,v) = 0$ then the subset $f(M)$ is an open subset of the sphere 
$$\{(t,x,0,z) \in \mink | \; -t^{2} + x^{2} + z^{2} = 1\},$$ 
so this surface is a totally geodesic open submanifold of the $2$-dimensional de Sitter space form, away from the set $u=v$. Thus, it is minimal in $\sitter$. 
\end{example}

 \begin{example}\label{100}
For each $w = u + iv \in \mathbb C$ \ let  
$$
\mu(u,v) = \frac{\sqrt{2}(1+i)}{4} e^{(v - u)}
$$
$$
x(u,v) = e^{ (u -  v) +i(v + u)} \ \ \ \ {\rm and} \ \ \ \ y(u,v) = - e^{(u -  v) +i( v + u)}
$$
Then we have an isotropic surface in $\sitter$ and the shape operator, with respect to the flat null coordinates $\{u,v\}$ is $\begin{bmatrix} 0 & 1 \\1 & 0 \end{bmatrix}.$ \ In fact, the functions $f(w) = \frac{\mu W + \overline \mu \overline W}{2}$ and  $\nu(w) = \frac{ \mu W - \overline \mu \overline W}{2i}$ take the form
$$
f(u,v) = \frac{\sqrt{2}}{2}(\sinh(v-u), - \sin(u+v), \cos(u+v), - \cosh(v-u)),
$$
$$
\nu(u,v) = \frac{\sqrt{2}}{2}(\sinh(v-u),  \sin(u+v), - \cos(u+v), - \cosh(v-u)).
$$
Then $\l f, f\r =1 = \l \nu, \nu\r$, $\l f, \nu \r = 0 = \l f_u, \nu\r = \l f_v, \nu\r$  and $f_u$, $f_v$ are lightlike vectors with $\l f_u, f_v\r = F = 1$.  So we are taking the basis $\{f, f_u, f_v, \nu\}$ of \ $\mathbb R^4_1$. 
\ Moreover $\l f_{uv}, \nu\r = 0$, which implies that the surface is minimal, so $\Phi =0$. Hence
 using formulas (\ref{alpha,beta minimal}), the real valued functions $\alpha$ and $\beta$ take the form
 $$
 \alpha = - \frac{\sqrt{2}}{4} e^{v-u} = - \beta.
 $$
 It is easy to see that the spherical and isotropic conditions (\ref{92}) and (\ref{89}) are satisfied. Finally we note that from Theorem (\ref{99}) the pair $(M, \nu)$ also represents a minimal timelike surface in $\mathbb S^3_1(1)$ with Gauss map given by $f(w)$, and whose isotropic conditions are given by formula (\ref{89}).
\end{example}

\section{When $\Phi = 0$ and $x,y$ are holomorphic functions satisfying the system(\ref{xuvyuv})}

In this last section we focus on surfaces with $\Phi = 0$ and for which $x, y\in \mathcal H(M)$ satisfy the system(\ref{xuvyuv}). In particular we show that in this case the functions $x$ and $y$ are related by a Mobius transformation in a complex variable and that the argument $\theta$ of the complex expression of the integration factor $\mu$ for the local expression $W(x,y)$,  should be  a  harmonic function in $M$. In particular, we give explicit formulas for $x$ and $y$ when we assume $(M, f)$ is a minimal isotropic surface in $\mathbb S^3_1(1)$ where $f$ is given by (\ref{standardequation}).  We also give, using the techniques developed in this paper,  the explicit construction of  families of timelike surface in $\mathbb S^3_1(1)$ whose $\Phi =0$. This example will be a generalization of Example (\ref{100}).

\begin{theor} \label{97}
Let $x(w)$ and $y(w)$ be two holomorphic functions from $M$ into $\complex$, such that $x - y \neq 0$ and  $x'y'\ne 0$. 
Since $x_{u} = x'$ and $x_{v} = i x'$, and the same is true for $y$, the system (\ref{xuvyuv}) for these functions becomes, after dividing by $i$ on both sides:
\begin{equation}\label{29} 
x'' = \frac{2 {x'}^{2}}{x - y} \; \; \mbox{ and } \; \; y'' = \frac{-2 {y'}^{2}}{x - y}.  
\end{equation}
Then, there exists a  M\"obius transformation 
$$M_{c}(z) = \frac{z}{cz - 1} \; \; \mbox{ where } \; \; c \in \overline{\complex}$$ 
such that, we have $y'(w) = M_{c}(x'(w))$ for each $w \in M$. Conversely, if $x(w)$ and $y(w)$ are related by 
$M_{c}(z)$ and $x(w)$ is a solution of the equation (\ref{29}) for $x$, then $y(w)$ is another solution  $y$, of the system (\ref{xuvyuv}). 
\end{theor}
\begin{proof}
Since 
$$\frac{x''}{x'^{2}} + \frac{y''}{y'^{2}} = \left(\frac{-1}{x'}\right)' + \left(\frac{-1}{y'}\right)' = 0 
\; \; \Longleftrightarrow \; \; \left(\frac{1}{x'}\right) + \left(\frac{1}{y'}\right) = c \in \complex$$ 
we obtain the family of relations $y' = M_{c}(x')$. For the converse, we assume that $y' = M_{c}(x')$ and $x$ satisfies (\ref{29}), then it follows that  $y''/(y')^{2} = -2/(x - y)$.
\end{proof}

\begin{corol}
Let $x(w),  \ y(w) \in \mathcal H(M)$ such that $x - y \neq 0$,  $x'y'\ne 0$, with $x$ and $y$ satisfying  equation (\ref{xuvyuv}). 
If $c = 0$ in the M\"obius transformation then $y' = -x'$, and if $c = \infty$ then $y' = 0$.
\end{corol}

\begin{example}
If $c=0$ this means $x' + y' = 0$,  which implies $x + y = 2a$ for $a \in \complex$. Hence taking $x - y = 2z$ we obtain $x = z + a$ and then equation (\ref{29}) become
$\frac{z''}{z'} = \frac{z'}{z}$. \ So  $\log z' = \log z + \log k = \log k z$  therefore  
$\frac{z'}{z} = k,$ for complex number $k$. Then, the solution of the system (\ref{29}) is    
$$x = a + e^{kw + b}\; \; \ \  {\rm and } \; \; \ \  y = a - e^{kw + b}$$ 
for complex numbers $a, b$ and $k \neq 0$. 
\end{example}

Now we obtain informations about the argument of the integration factor $\mu$.  Since $\vert \mu \vert = 1/ \vert x - y\vert$ the polar form of this 
function is 
$$\mu(w) = \frac{e^{i \theta(w)}}{\vert x(w) - y(w) \vert}.$$ 

\begin{lemma}
For $x,y \in \mathcal{H}(M)$, the spherical condition (\ref{92}) for the polar form of $\mu$ is
\begin{equation}\label{95}
\theta_{w} = \frac{i}{2} \; \frac{x' + y'}{x - y}. 
\end{equation}
Therefore, the real valued function $\theta$ is harmonic in $M$. 
\end{lemma}
\begin{proof}
Since $x_w = x'$, \ $x_{\overline w} = 0$  \ and 
$$\log \mu = i \theta - \frac{1}{2} \log (x - y) - \frac{1}{2} \log (\overline{x} - \overline{y}) 
\; \; \mbox{ then } \; \; \ \ \frac{\mu_{w}}{\mu} = i \theta_{w} -  \frac{1}{2}\frac{x' - y'}{x - y}.$$ 
Since same equations hold for $y$, we obtain from equations (\ref{92}) 
$$\frac{-x'}{x - y} = i \theta_{w} - \frac{1}{2} \; \frac{x'}{x - y} + \frac{1}{2} \; \frac{y'}{x - y},$$ 
which implies equation (\ref{95}).
\end{proof}

Here we recall that we are assuming that $f_u$ is a multiple of $L(x)$ and $f_v$ is a multiple of $L(y)$.
\begin{lemma}
For $x,y \in \mathcal{H}(M)$, the isotropic condition (\ref{89}) corresponds to the equations 
\begin{equation}\label{96}
\Re \left( e^{i \theta} \frac{x'}{x - y} \right) = 0 \; \; \ \ {\rm and } \; \; \ \ 
\Re \left( e^{-i \theta} \frac{iy'}{x - y} \right) = 0.
\end{equation}
\end{lemma}
\begin{proof}
It follows from equations (\ref{89}) since $|x-y|$ is real,  $x_{u} = x'$ and $y_{v} = i y'$. 
Indeed, equation (\ref{89}) say that $\frac{e^{i \theta} \; x'}{x - y}$  and $\frac{i e^{-i \theta} \; y'}{x - y}$ are 
imaginary valued functions.  
\end{proof}
\begin{corol}
For $x,y \in \mathcal{H}(M)$, the equations (\ref{96}) mean that 
$$arg\left(\frac{x'}{x - y}\right) = -\theta \pm \frac{\pi}{2}+2 k \pi \; \; \ \  {\rm and } \; \;  \ \ 
arg\left(\frac{y'}{x - y}\right) = \theta\pm \pi  + 2k \pi, \; \; \mbox{ for } \; \; k \in \mathbb{Z}.$$ 
\end{corol}

This last corollary  says that the function $\theta$ carries quite a bit of information about the holomorphic functions 
$x'$, $y'$ and $x - y$.  

\begin{theor}[Necessity] 
Assume that $(M,f)$ is a minimal  isotropic surface into $\sitter$, such that $\theta$ is a non-constant real 
valued harmonic function.  In addition we suppose that $x$ and $y$ are holomorphic functions where $f$ is given by equation (\ref{standardequation}). Then, 
there exists constants $k,c \in \complex \setminus \{0\}$ such that 
\begin{equation}\label{98}
x(w) = \frac{1}{c} \int_{w_{0}}^{w} (1 + ke^{\psi(\xi)})d \xi \; \; \ \ {\rm and } \; \; \ \ 
y(w) = \frac{1}{ck} \int_{w_{0}}^{w} (k + e^{-\psi(\xi)})d \xi.
\end{equation} 
where the harmonic function $\psi$ is given by \; \; \; \; 
\begin{equation}
\psi(w) = \theta(w_{0}) - 4i \int_{w_{0}}^{w} \theta_{w}(\xi) d \xi.
\end{equation}
\end{theor}

\begin{proof}
By Theorem (\ref{97}) we have that there exists a Mobius transformation $M_c$ such that $y' = M_c(x')$. Then $y' (cx' - 1) = x'$. Hence
 $x' + y' = cx'y'$ for $c \neq 0$. 
 
 Now, from equations (\ref{29}) 
we obtain from equation (\ref{95}): \; \; \; \; 
$$\frac{x''}{x'} - \frac{y''}{y'} = 2 \frac{x' + y'}{x - y} = -4i \theta_{w} =: \psi_w.$$ 
Then, we have the system 
$$x' + y' = cx'y' \; \; \mbox{ and } \; \; \frac{x'}{y'} = ke^{\psi},$$ 
since the logarithmic derivative  $x'/y'$ equals $\psi_w$. From 
these two equations it follows $cx'= 1 + ke^{\psi}$ and $kcy' = k + e^{-\psi}$ therefore we get expressions (\ref{98}). Then, from the fact that $\theta$ is harmonic function, it follows  immediately that $\psi$ is also a harmonic function. 
\end{proof}

In the last example we construct  families of isotropic surfaces in $\mathbb S^3_1(1)$ using the technique described above. In fact,
 
\begin{example}
For complex numbers $c$ and $k \neq 0$ and taking $0 \neq r \in \real$ let us define for each $w = u + iv \in \complex$:  
\begin{align*}
\mu(u,v) = \frac{\sqrt{2}(1 + i)}{4\vert k \vert} e^{r( v -  u)}\\
x(u,v) = c + k e^{(1+i) r (u +iv)}  \\ 
y(u,v) = c - k e^{(1+i) r (u +iv)}
\end{align*}
then these data give us families of isotropic surfaces on $\sitter$. 

\vspace{0.2cm}

In fact, we begin by assuming $x, y \in \mathcal H(M)$ such that  $x + y = 2c$ \ and \ $x - y = 2z =  2ke^{aw}$, where $c, a \in \mathbb C$, and 
 $k \in \mathbb C - \{0\}$. Then we see first that $z'/z = a$. 

\vspace{0.1cm}
Now we look for the function $\mu$ satisfying the spherical and isotropic equations (\ref{92}), (\ref{89}), to obtain an isotropic immersion in $\mathbb S^3_1(1)\subset \mathbb R^4_1$. 

 From equation 
(\ref{92}) we obtain 
$$\frac{\mu_{w}}{\mu} = \frac{-x_{w}}{x - y} = \frac{\overline{\mu}_{w}}{\overline{\mu}} =  \frac{y_{w}}{x - y}=  \frac{-a}{2},$$ 
because $x_{w} = x' = -y'$. Now, since we need $\vert \mu \vert = 1/\vert 2k e^{aw} \vert$ with $0 \neq k \in \complex$, we take then 
$$\mu(u,v) = \frac{e^{i \theta}}{2 \vert k \vert} e^{-\Re(aw)}.$$
We note that since $\frac{\mu_{w}}{\mu} = \frac{-a}{2} $ it follows that $\theta_{w} = 0$, which implies that $\theta \in \mathbb C$ is constant. 

Now, since $x_w = x_u = x'$, the second equation of (\ref{89}) says that $\Re(\mu \frac{a}{2}) = \Re (\mu \frac{x_u}{x-y}) =0$, so $\mu \frac{a}{2}$ is imaginary.  Since $y_{v} = i y' = -i x'$ the first equation of  (\ref{89}) says that 
$\overline{\mu} i \frac{a}{2}$ is also imaginary.  Taking $e^{i \theta} = p$, we obtain that 
\begin{equation}\label{99} 
p a = -\overline{p} \; \overline{a} \; \; \ \ {\rm and } \; \; \ \ \overline{p} a =  p \overline{a}.
\end{equation}
This last implies that 
$(\frac{a}{\overline a})^2 = -1 = (\frac{p}{\overline p})^2$. \ Then from $p^2= - (\overline p)^2$, we find that for some real $b$, $p=b (1 \pm i )$.  Analogously we get that for some real $r\ne 0$,  $a=r(1\pm i).$
Finally, since $\vert p \vert = 1$ and remembering that $a$ and $p$  have to satisfy equation (\ref{99}),  we choose from a set of four possible solutions for $p^{4} = 1$, the following values: 
$$p = e^{i\pi/4} = \frac{\sqrt{2}}{2}(1 + i) \; \; \ \ {\rm  and } \; \; \ \ a = r(1 + i).$$ 
Hence the equations enunciated in the beginning of this example follows from this choice for $\mu(u,v)$. 
\end{example}

{\bf Acknowledgments}
The first author's research was supported by Projeto Tem\'atico Fapesp n. 2016/23746-6. S\~ao Paulo. Brazil.

\end{document}